\documentclass[11pt]{amsart}
\usepackage{amsmath,amsthm,amssymb,amsfonts, esint}
\usepackage[nobysame,abbrev,alphabetic]{amsrefs}
\usepackage{url}
\usepackage{ esint }
\usepackage{graphicx}
\usepackage{graphicx}
\usepackage{color}
\usepackage{times}
\usepackage{cite}
\usepackage{enumerate,latexsym}
\usepackage{latexsym}
\usepackage{amsmath,amssymb}
\usepackage{graphicx}

\usepackage{amsthm}
\usepackage{verbatim}
\usepackage{tikz,tikz-cd}
\usepackage{url}
\usepackage{graphicx}
\usepackage{color}
\usepackage{times}
\usepackage{cite}
\usepackage{enumerate,latexsym}
\usepackage[bottom]{footmisc}
\usepackage{latexsym}
\usepackage{amsmath,amssymb}
\usepackage{graphicx}

\usepackage{amsthm}
\usepackage{verbatim}
\usepackage[top=2cm, bottom=2.8cm, left=3cm, right=3cm]{geometry}%
\newtheorem{thm}{Theorem}[section]
\newcommand{\bt}{\begin{thm}}
\newcommand{\et}{\end{thm}}

\newtheorem{conj}[thm]{Conjecture}

\newtheorem{cor}[thm]{Corollary}  
\newcommand{\bc}{\begin{cor}}
\newcommand{\ec}{\end{cor}}

\newtheorem{lem}[thm]{Lemma}   
\newcommand{\bl}{\begin{lem}}
\newcommand{\el}{\end{lem}}

\newtheorem{prop}[thm]{Proposition}
\newcommand{\bp}{\begin{prop}}
\newcommand{\ep}{\end{prop}}

\newtheorem{defn}[thm]{Definition}
\newcommand{\bd}{\begin{defn}}    
\newcommand{\ed}{\end{defn}}

\newtheorem{rmrk}[thm]{Remark}

\newcommand{\R}{\mathbb{R}}
\newcommand{\N}{\mathbb{N}}

\newcommand{\lp}{\left (}
\newcommand{\rp}{\right )}

\newcommand{\Tor}{\mathbb{T}}     %Torus%

\newcommand{\Fto}{\stackrel {\mathcal{F}}{\longrightarrow} }

\newcommand{\VFto}{\stackrel {\mathcal{VF}}{\longrightarrow} }

\DeclareMathOperator{\Vol}{Vol}
\DeclareMathOperator{\Diam}{Diam}

\def\Xint#1{\mathchoice
{\XXint\displaystyle\textstyle{#1}}%
{\XXint\textstyle\scriptstyle{#1}}%
{\XXint\scriptstyle\scriptscriptstyle{#1}}%
{\XXint\scriptscriptstyle\scriptscriptstyle{#1}}%
\!\int}
\def\XXint#1#2#3{{\setbox0=\hbox{$#1{#2#3}{\int}$ }
\vcenter{\hbox{$#2#3$ }}\kern-.6\wd0}}

\def\dashint{\Xint-}
\title{Almost non-negative scalar curvature on Riemannian manifolds conformal to tori}

\author{Brian Allen}
\address[Brian Allen]{University of Hartford}
\email{brianallenmath@gmail.com}

\begin{document}
\date{\today}
\begin{abstract}
In this article we reduce the geometric stability conjecture for the scalar torus rigidity theorem to the conformal case via the Yamabe problem. Then we are able to prove the case where a sequence of Riemannian manifolds is conformal to a uniformly controlled sequence of flat tori and satisfies the geometric stability conjecture. We are also able to handle the case where a sequence of Riemannian manifolds is conformal to a sequence of constant negative scalar curvature Riemannian manifolds which converge to a flat torus in $C^1$. The full conjecture from the conformal perspective is also discussed as a possible approach to resolving the conjecture.
\end{abstract}
\maketitle
\section{Introduction} 

The Scalar Torus Rigidity Theorem says that a Riemannian manifold with non-negative scalar curvature which is diffeomorphic to a torus must be isometric to a flat torus. In 1979, Schoen and Yau \cite{SY79} were able to prove the Scalar Torus Rigidity Theorem using minimal surface techniques which can now be extended to higher dimensions. In 1980, Gromov and Lawson \cite{GL80} gave a proof in all dimensions using the Lichnerowicz formula. 

In 2014, Gromov \cite{GroD} suggested that there should be a geometric stability result corresponding to the Scalar Torus Rigidity Theorem. More specifically, a sequence of Riemannian manifolds with almost non-negative scalar curvature, which are diffeomorphic to flat tori, combined with appropriate compactness conditions should converge to flat tori. Below we give a version of the conjecture of Gromov in the author's words.

\begin{conj}\label{MainConjecture}
Let $M_j = (\mathbb{T}^n, g_j)$, $n \ge 3$ be a sequence of Riemannian manifolds such that 
\begin{align} \label{HypothesisConjecture}
R_{g_j} \ge -\frac{1}{j}, \,\,\, V_0 \le \Vol(M_j) \le V^0 \,\,\, \text{and} \,\,\, \Diam(M_j) \le D_0  ,
\end{align}
where $R_{g_j}$ is the scalar curvature. 
If no bubbling occurs along the sequence then there is a subsequence of $M_j$ converging in the volume preserving intrinsic flat sense to a flat torus
\begin{align}
M_{j_k} \VFto \mathbb{T}^n.
\end{align} 
\end{conj}

Without the upper bound on diameter it is possible for a sequence to converge to a cylinder or even Euclidean space. Without the lower  bound on volume a sequence could collapse to a circle or a point, even if one assumes nonnegative sectional curvature. We also note that one needs to make precise the condition which one will add to prevent bubbling from occurring along the sequence because otherwise a sequence could converge to a torus with a sphere attached. In the main theorems below we will see the author's choice for avoiding bubbling in the $n-$dimensional case. In addition, in section \ref{sect:Reduction} the author discusses how Conjecture \ref{MainConjecture} can be reduced to the conformal case by taking advantage of the resolution of the Yamabe problem.

The first results in this direction are given by Gromov in 2014 \cite{GroD} where he assumes that a sequence of tori with almost non-negative scalar curvature converges in the $C^0$ sense then one can show $C^0$ convergence of the sequence to a flat torus. Subsequently, Bamler in 2016 \cite{Bamler-16}  was able to show the same result using Ricci flow.

There has also been progress on the conjecture, made precise by Sormani \cite{Sormani-scalar}, where the MinA condition is what rules out bubbling in Conjecture \ref{MainConjecture} for dimension $n=3$. In \cite{AHMPPW}, the author, Hernandez-Vazquez, Parise, Payne, and Wang study the warped product case of the conjecture by Sormani where we were able to show uniform, GH, and SWIF convergence when the sequence is either a doubly warped product or a singly warped product. In \cite{PKP19}, Cabrera Pacheco, Ketterer, and Perales study the graphical case of the Sormani conjecture and by adding some conditions which are motivated by the work of Huang and Lee \cite{HL} and the work of Huang, Lee, and Sormani \cite{HLS} they were able to show volume preserving intrinsic flat convergence for an interesting class of graphs over flat tori which satisfy Conjecture \ref{MainConjecture}. Recently, Lee, Naber, and Neumayer \cite{LNN} prove a version of Gromov's conjecture under $d_p$ convergence where a lower bound on entropy is assumed with many interesting examples for dimension $n \ge 4$ given in Section 9.

In this paper we study Conjecture \ref{MainConjecture} using the conformal geometry implied by the resolution of the Yamabe problem. In particular we resolve the case where the sequence satisfying the hypotheses of Conjecture \ref{MainConjecture} is conformal to a uniformly controlled sequence of flat tori. We note that the lower bound on volume in Conjecture \ref{MainConjecture} is necessary to prevent collapsing, but is not as strong as the MinA condition, and specifically is not strong enough to prevent bubbling. In the conformal case we require an upper bound on an integral of a negative power of the conformal factor which will prevent collapsing. This condition is weaker than a uniform lower bound on the conformal factor but stronger than a lower bound on volume.

Bubbling is the phenomenon by which a cylindrical bridge between a torus minus a ball and a bubble with the toplogy of a ball is formed with almost nonnegative scalar curvature. Since the bubble is separated from the rest of the torus thorugh the arbitrarily thin neck it is not subject to the same rigidity like the rest of the torus. Hence the bubble is free to take on any geometry it wishes as long as the scalar curvature is almost non-negative and in the limit one expects the cylindrical bridge to contract to a point leaving the bubble to be attached to a flat torus at a point.  This is why one needs to either rule out these bubbles in a precise statement of the conjecture or locate and cut out these bubbles somehow. 

From the conformal point of view we will see that bubbling presents itself through singularities of the conformal factor where volume concentrates along a set of measure zero as in Example 3.5 of the author and Sormani \cite{Allen-Sormani-2}. In that example the authors showed that the conformal factor is blowing up at a point, bounded in $L^n$, and not bounded in $L^p$ for $p > n$. Geometrically the authors show that the sequence converges to a flat torus with a disk attached so that the boundary of the disk is identified with a point on the torus. Analytically, what is happening in this example is a consequence of the fact that bounded $L^1$ functions are not weakly compact in $L^1$ and hence may converge to a measure instead. So the volume of the conformal metric, given by integrating $e^{nf_j}$, is a bounded $L^1$ function which is concentrating along a set of measure zero and hence converging to a measure. By the Dunford-Pettis Theorem it is well known that if one adds a equiintegrable assumption to a sequence of functions bounded in $L^1$ then the sequence is weakly compact in $L^1$. Hence, we rule out bubbling through a equiintegrable type condition \eqref{UniformIntegrability} which will not allow bubbling to occur through the concentration of volume on a set of measure zero.

 We do note that condition \eqref{UniformIntegrability} does allow for arbitrarily many bubbles to form along a sequence as long as their overall volume shrinks to zero in the limit so that the hypotheses of Theorem \ref{MainThm} are satisfied. Also, this condition allows for arbitrarily many splines to form along a sequence, as in Example 3.7 and 3.8 of the author and Sormani \cite{Allen-Sormani-2}, which is the main prototypical examples of Conjecture \ref{MainConjecture}. In \cite{Allen-Sormani-2} the authors show that Example 3.7 and 3.8 are bounded in $L^p$ for $p >n$ and hence by Lemma \ref{L^3 Agrees with C^0 2} we see that they satisfy the equiintegrable condition \eqref{UniformIntegrability}.  This condition can be imposed for any dimension and can be used for the full conjecture when viewed through the conformal point of view. This condition shows up in both Theorem \ref{MainThm} and Theorem \ref{MainThm3}.

\begin{thm}\label{MainThm}
For a sequence of Riemannian $n-$manifolds $M_j=(\mathbb{T}^n,g_j)$, $n \ge 3$ satisfying 
 \begin{align} \label{HypothesisMainThm}
R_{g_j} \ge -\frac{1}{j}, \,\,\,  \,\,\, \Diam(M_j) \le D_0, \,\,\,  \,\,\, \Vol(M_j) \le V_0,
\end{align}
which is conformal to a flat torus $\mathbb{T}_j^n$, i.e.  $g_j = e^{2f_j} g_{0,j}$ where $\mathbb{T}^n_j=(\mathbb{T}^n,g_{0,j})$, so that if $\theta_i \in [0,2\pi], 1 \le i \le n,$ are coordinates on $\Tor^n$ then
\begin{align}
0<c\le|(g_{0,j})_{ik}|\le C<\infty, \quad (g_{0,j})_{ik} \in \R,\quad 1 \le i,k\le n,
\end{align}
\begin{align}
\int_{\mathbb{T}^n} e^{-2f_j} d V_{g_{0,j}} \le C,
\end{align} 
and
\begin{align}\label{UniformIntegrability}
\exists q > 0 \text{ so that } \forall E \subset \mathbb{T}^n \text{ measurable } \Vol_{g_j}(E) \le C \Vol_{g_{0,j}}(E)^q,
 \end{align}
then there exists a subsequence so that $M_k$ converges in the volume preserving intrinsic flat sense to a flat torus
\begin{align}
M_k \VFto \bar{\mathbb{T}}_{\infty}^n,
\end{align}
where $\bar{\mathbb{T}}_{\infty}^n = (\Tor^n,\bar{g}_{\infty}=c_{\infty}^2g_{\infty})$,
\begin{align}
c_{\infty}^2&=\displaystyle\lim_{k \rightarrow \infty}(\overline{e^{-f_k}})^{-2}= \lim_{k \rightarrow \infty} \left( \dashint_{\mathbb{T}^n} e^{-f_k}dV_{g_{0,j}} \right)^{-2},
\end{align}
and 
\begin{align}
0<c\le|(g_{\infty})_{ik}|\le C<\infty, \quad (g_{\infty})_{ik} \in \R,\quad 1 \le i,k\le n.
\end{align}
\end{thm}

We can also rule out a possibility in the case where the sequence of Riemannian manifolds satisfying conjecture \ref{MainConjecture} is conformal to a manifold with constant negative scalar curvature.

\begin{thm}\label{MainThm2}
If $M_j = (\mathbb{T}^n, g_j)$, $n \ge 3$ is a sequence of Riemannian manifolds such that 
\begin{align} \label{HypothesisConjecture}
R_{g_j} \ge -\frac{1}{j}\,\,\, \text{and} \,\,\,  \Vol(M_j) \le V_0    ,
\end{align}
so that $M_j$ is conformal to $M_{0,j}$ where the scalar curvature $R_{g_{0,j}} = -1$ then 
\begin{align}
Vol(M_{0,j}) \le \frac{V_0}{j^{3/2}}.
\end{align}
In other words, if $M_j$ is conformal to $\tilde{M}_{0,j}$ with $\tilde{R}_{g_{0,j}} <0$ and $Vol(\tilde{M}_{0,j})=1$ then
\begin{align}
0 > \tilde{R}_{g_{0,j}} \ge -\frac{V_0^{2/3}}{j}. 
\end{align}
\end{thm}

In our last main theorem we notice that if the sequence of background metrics of constant negative scalar curvature converge to a flat torus in $C^1$ then we can also guarantee volume preserving intrinsic flat convergence.

\begin{thm}\label{MainThm3}
For a sequence of Riemannian $n-$manifolds $M_j$, $n \ge 3$ satisfying  satisfying the \begin{align} \label{HypothesisMainThm2}
R_{g_j} \ge -\frac{1}{j},   \,\,\, \Diam(M_j) \le D_0, \,\,\,  \,\,\, \Vol(M_j) \le V_0,
\end{align}
 so that $M_j$ is conformal to $\tilde{M}_{0,j}=(M,\tilde{g}_{0,j})$, a metric with constant negative scalar curvature and unit volume, i.e.  $g_j = e^{2f_j} g_{0,j}$, so that
\begin{align}
\tilde{g}_{0,j} \rightarrow g_0 \text{ in } C^1,
\end{align}
 where $g_0$ is a flat torus where $\mathbb{T}^n_0=(\mathbb{T}^n,g_0)$, so that if $\theta_i \in [0,2\pi], 1 \le i \le n,$ are coordinates on $\Tor^n$ then
\begin{align}
0<c\le|(g_0)_{ik}|\le C<\infty, \quad (g_0)_{ik} \in \R,\quad 1 \le i,k\le n,
\end{align}
\begin{align}
\int_{\mathbb{T}^n} e^{-2f_j} d V_{\tilde{g}_{0,j}} \le C,
\end{align} 
and
\begin{align}
\exists q > 0 \text{ so that } \forall E \subset \mathbb{T}^n \text{ measurable }\Vol_{g_j}(E) \le C \Vol_{g_0}(E)^q,
 \end{align}
then there exists a subsequence so that $M_k$ converges in the volume preserving intrinsic flat sense to a flat torus
\begin{align}
M_k \VFto \bar{\mathbb{T}}_0^n,
\end{align}
where $\bar{\mathbb{T}}_0^n = (\Tor^n,\bar{g}_0=c_{\infty}^2g_0)$,
\begin{align}
c_{\infty}^2&=\displaystyle\lim_{k \rightarrow \infty}(\overline{e^{-f_k}})^{-2}= \lim_{k \rightarrow \infty} \left( \dashint_{\mathbb{T}^n} e^{-f_k}dV_{\tilde{g}_{0,j}} \right)^{-2}.
\end{align}
\end{thm}

\begin{rmrk}\label{remark:MainThm2Discussion}
Note that in particular Theorem \ref{MainThm2} shows that it is not possible to have a sequence of manifolds $M_j$ satisfying Conjecture \ref{MainConjecture} which is conformal to a fixed background metric of constant negative scalar curvature. In particular, Theorem \ref{MainThm2} implies that the background sequence of Yamabe metrics, $g_{0,j}$, with $\Vol(M_{0,j})=1$ must be a sequence which converges to the Yamabe constant of the torus which is zero. One possibility for $g_{0,j}$ is that it is a sequence of uniformly controlled flat tori which is explored in Theorem \ref{MainThm}. Another possibility is that $g_{0,j}$ converges to a flat torus in $C^1$ which is explored in Theorem \ref{MainThm3}. 

 We also note that the main theorems of this paper show that the remaining cases of Conjecture \ref{MainConjecture} from the conformal perspective arise from the possibility of a sequence $M_j$ being conformal to a non-uniformly controlled sequence of flat tori $M_{0,j}$ or conformal to a sequence $M_{0,j}$ of Riemannian manifolds with $\Vol(g_{0,j})=1$ and $R_{g_{0,j}} \nearrow 0$ which is not uniformly controlled or which does not converge in $C^1$ to a flat torus. This is because the resolution of the Yamabe problem implies that any metric satisfying Conjecture \ref{MainConjecture} must be conformal to a metric with constant scalar curvature $\le 0$ (See section \ref{sect:Reduction}) combined with Theorem \ref{MainThm2}. 
 
 The difficulty of the remaining cases seems to be that the background sequence $M_{0,j}$ could be degenerating  whereas  the sequence $M_j$ could not be degenerating and hence satisfies Conjecture \ref{MainConjecture}. This would make the analysis quite delicate but gives a possible approach to the full conjecture. See the work of Anderson \cite{Anderson97, Anderson99} for a discussion of the convergence of degenerating constant scalar curvature sequences as well as the discussion at the end of section \ref{sect:Reduction}. 
\end{rmrk}

In section \ref{sec: background}, we give the reader intuition for volume preserving intrinsic flat convergence through a new estimate by the author, Perales, and Sormani \cite{Allen-Perales-Sormani}, as well as other standard definitions, and state the main theorem of \cite{Allen-Perales-Sormani} which will be used to complete the proof of the main theorems. 

In section \ref{sect:Reduction}, we remind the reader of the Yamabe problem and discuss its relation to the conformal perspective of Conjecture \ref{MainConjecture}. The proof of Theorem \ref{MainThm2} is given in this section as well by noticing a consequence of the scalar curvature formula for metrics conformal to metrics with constant negative scalar curvature.

In section \ref{sect:ConfFlatCase}, we obtain important estimates by studying the PDE inequality which follows from the scalar curvature formula for conformal metrics. We obtain important uniform bounds, $L^p$ bounds, and Sobolev convergence which allows us to apply the main theorem of \cite{Allen-Perales-Sormani} to finish the proof of Theorem \ref{MainThm}.

\noindent{\bf Acknowledgements:} This research was funded in part by NSF DMS - 1612049. The author would like to thank Christina Sormani for funding provided during the 2020 Virtual Workshop on Ricci and Scalar Curvature. I would also like to thank the organizers of this workshop Christina Sormani, Guofang Wei, Hang Chen, Lan-Hsuan Huang, Pengzi Miao, Paolo Piazza, Blaine Lawson, and Richard Schoen for the opportunity to give a plenary talk.

\section{Background} \label{sec: background}
\noindent In this section, we review definitions and theorems that will be used throughout the paper.

\subsection{Notation}
We denote the volume and diameter of a Riemannian manifold $M$ as $\Vol(M)$ and $\Diam(M)$. For a measurable subset $E \subset \mathbb{T}^n$ we denote the measure of $E$ with respect to a Riemannian metric $g$ as $\Vol_g(E)$ and the corresponding measure as $dV_g$. The distance function associated to a Riemannian metric $g$ will be denoted $d_g$, the scalar curvature as $R_g$, the gradient as $\nabla^g$, and the Laplacian as $\Delta^g$. $\mathbb{T}^n$ will stand for the flat torus with metric $g_{\Tor^n}=d\theta_1^2+...+d\theta_n^2$, $\theta_i \in [0,2\pi]$, $1 \le i \le n$. For a flat torus $\bar{g}$ defined on $\Tor^n$ with coordinates $\theta_i \in [0,2\pi]$, $1 \le i \le n$ with coordinate vectors $\partial_{\theta_i}$, $1 \le i \le n$ we let $\bar{g}_{ik}=\bar{g}(\partial_{\theta_i},\partial_{\theta_k})\in \R$. We denote a metric ball in a Riemannian manifold $M$ as $B_M(x,r)$, $x\in M$, $r>0$.

\subsection{Sormani-Wenger Intrinsic Flat Convergence}\label{subsect:BackgroundSWIF}

The intrinsic flat distance  was defined by Sormani and Wenger \cite{SW-JDG} on a large class of metric spaces called integral current spaces. In particular, oriented, compact Riemannian manifolds are integral current spaces and the intrinsic flat distance provides a notion of convergence for sequence of Riemannian manifolds which is weaker than Gromov-Hausdorff distance. In fact, there are examples where the Gromov-Hausdorff limit does not even exist but for which the intrinsic flat limit is a Riemannian manifold (See Example A.7 of \cite{SW-JDG}, Example 2.8 of \cite{Allen-Perales-Sormani}, Example 3.8 of \cite{Allen-Sormani-2}).

For two Riemannian $n$-manifolds $M,N$ we denote the intrinsic flat distance by $d_{\mathcal{F}}(M,N)$ and we say that $M_j$ converges to $N$ in the intrinsic flat sense, denoted $M_j \Fto N$ if $d_{\mathcal{F}}(M_j,N)\rightarrow 0$. We say that $M_j$ converges to $N$ in the volume preserving intrinsic flat sense, denoted $M_j \VFto N$ if $d_{\mathcal{F}}(M_j,N)\rightarrow 0$ and $\Vol(M_j) \rightarrow \Vol(N)$. Instead of giving the  definition of intrinsic flat convergence, which would require many details about integral current spaces, we provide a recent estimate of intrinsic flat convergence for Riemannian manifolds satisfying special properties which gives the geometric intuition behind the intrinsic flat convergence that is relevant to this paper. For the definition of Sormani-Wenger intrinisic flat convergence see \cite{SW-JDG} and for an inutitive survey of various notions of convergence see \cite{Sormani-scalar}.

\begin{thm}[B.A., R. Perales, C. Sormani \cite{Allen-Perales-Sormani}]
Let $M$ be a compact, connected, oriented manifold without boundary, $M_j=(M,g_j)$, $M_0=(M,g_0)$ Riemanian manifolds with $\Diam(M_j) \le D$, $\Vol(M_j) \le V$, and
\begin{align}
g_j(v,v) \ge g_0(v,v), \quad \forall v \in T_pM, p\in M.
\end{align}
Let $W_j \subset M$ be a measurable set and assume $\exists \delta_j > 0$ such that if $d_j,d_0$ are the distance function for $M_j,M_0$ then
\begin{align}
d_{g_j}(p,q) \le d_{g_0}(p,q) + 2 \delta_j, \quad \forall p,q \in W_j.
\end{align}
If
\begin{align}
\Vol_{g_j}(M \setminus W_j) \le V_j, \quad h_j \ge \sqrt{\delta_j D + \delta_j^2},
\end{align}
then 
\begin{align}
d_{\mathcal{F}}(M_j,M_0) \le 2 V_j + h_j V.
\end{align}
\end{thm}

This theorem shows us that if the distance on $M_j$ is always longer than on $M_0$ and one can also control $d_j$ from above by $d_0$ on a set of large volume then one can estimate the intrinsic flat distance between Riemannian manifolds. In fact, if the volume of $M_j \setminus W_j \rightarrow 0$ and $h_j \rightarrow 0$ then one can show that the intrinisc flat distance converges to zero. The intuition being that $M_j$ is allowed to measure distances larger than $M_0$ as long as that only happens on a set of small measure. This should be contrasted with Gromov-Hausdorff or $C^{k,\alpha}$ convergence of Riemannian manifolds where this is not allowed.

\subsection{Contrasting and Relating Notions of Convergence}\label{subsect:BackgroundRelatingContrasting}

In \cite{Allen-Perales-Sormani}, the author, Perales, and Sormani give a theorem which provides hypothesis on a sequence of Riemannian manifolds which can be obtained through natural geometric analysis estimates. When these hypotheses are obtained the theorem guarantees volume preserving intrinsic flat convergence. Here we state Corollary 5.1 of \cite{Allen-Perales-Sormani} which is the version of the main theorem which we will aim to use in this paper.

\begin{thm}[B.A., R. Perales, C. Sormani \cite{Allen-Perales-Sormani}]\label{VolSWIFConv}
Let $M_j=(M^n,g_j)$, $M_0=(M^n,g_0)$ be compact, connected, oriented smooth Riemannian $n$-manifolds without boundary. If 
\begin{align}\label{MetricLowerHyp}
\left(1 - \frac{C}{j}\right ) g_0(v,v) \le g_j(v,v) \qquad \forall v \in T_pM, p \in M,
\end{align}
\begin{align}
\Diam(M_j) \le D_0,
\end{align}
and 
\begin{align}
\Vol(M_j) \rightarrow \Vol(M_0)\label{VolHyp}
\end{align}
then
\begin{align}
M_j &\VFto  M_0.
\end{align}
\end{thm}

Our goal in Section \ref{sect:ConfFlatCase} is to obtain the estimates necessary to apply this theorem. In \cite{Allen-Sormani,Allen-Sormani-2} the author and Sormani give warped product and conformal examples which illustrate the necessity of the hypotheses of this theorem. In particular, Example 3.4 of \cite{Allen-Sormani} and Example 3.1 of \cite{Allen-Sormani-2} show that \eqref{MetricLowerHyp} of Theorem \ref{VolSWIFConv} is essential. If one removes \eqref{MetricLowerHyp} then one can construct examples where short cuts are formed along sets of small measure so that the sequence does not converge to a Riemannian manifold. In Example 3.2 of \cite{Allen-Sormani-2} we see that one cannot expect a stronger notion of convergence under these hypotheses and Example 3.3 shows that pointwise convergence on a dense set is not a sufficient replacement for volume convergence. In Examples 3.4-3.8 of \cite{Allen-Sormani-2} the importance of volume convergence and its relationship to $L^p$ convergence for $p \ge n$ is shown for sequences conformal to flat tori. In particular, Example 3.5 shows how one should expect bubbling to occur in the conformal case under similar conditions as Conjecture \ref{MainConjecture} without the scalar curvature assumption. For the conformal case we will show the convergence in $L^n$ norm of the conformal factor $e^{f_j}$ which implies volume convergence when combined with \eqref{MetricLowerHyp} and Lemma 4.3 of \cite{Allen-Sormani-2}.

\subsection{Scalar Curvature of Conformal Metrics}\label{subsect:BackgroundScalar}
For the metric $g_j = e^{2f_j} g_0$ on $M^n$ we find the following equation for the scalar curvature:
\begin{align}
R_{g_j} =e^{-2f_j} \lp R_{g_0} -2(n-1)\Delta^{g_0} f_j -(n-2)(n-1) |\nabla^{g_0} f_j|^2\rp.\label{ScalarFormula}
\end{align}
See Lee and Parker \cite{LeePar87} for a discussion of scalar curvature formulas under conformal changes where the reader should be careful that their definition of Laplacian is the negative of ours.

\subsection{Scalar Curvature of Sequences Conformal to Flat Tori}\label{subsect:ConfFlatCase}

In this subsection we study sequences which satisfy Conjecture \ref{MainConjecture} and which are conformal to a flat torus.  We begin by deriving a family of elliptic inequalities which follow from \eqref{ScalarFormula}.
\begin{lem}\label{SobolevConvergence1}
Let $M_j$ be a sequence of Riemannian $n-$manifolds, $n \ge 3$ conformal to a flat torus $M_{0,j} = (\mathbb{T}^n,g_{0,j})$ 
 then we find that 
\begin{align}
 \frac{2}{\alpha} \Delta^{g_{0,j}} e^{\alpha f_j} +(n-2 - 2 \alpha) |\nabla^{g_{0,j}} f_j|^2 e^{\alpha f_j} &\le \frac{e^{(2+\alpha)f_j}}{(n-1)j}\label{generalPDESatisfiedf_j}
\end{align}
for $\alpha \in \R$ and hence
\begin{align}
\int_{\mathbb{T}^n}|\nabla^{g_{0,j}} f_j|^2 dV_{g_{0,j}} &\le \frac{V_0^{\frac{2}{n}}\Vol(M_{0,j})^{\frac{n-2}{n}}}{(n-1)j}, \label{FirstSobolevEstimate}
\\\int_{\mathbb{T}^n} |\nabla^{g_{0,j}} e^{-f_j}|^2  dV_{g_{0,j}} &\le \frac{Vol(M_{0,j})}{j(n-1)(n+2)},\label{SecondSobolevEstimate}
\\ \int_{\mathbb{T}^n}|\nabla^{g_{0,j}} e^{\frac{\alpha}{2}f_j}|^2dV_{g_{0,j}} &\le \frac{\alpha^2V_0^{\frac{2+\alpha}{n}} \Vol(M_{0,j})^{\frac{n-2-\alpha}{n}}}{4(n-2-2\alpha)(n-1)j},\label{ThirdSobolevEstimate}
\end{align}
for $\alpha \in \lp 0,\frac{n-2}{2}\rp$.
\end{lem}

\begin{proof}
\begin{align}\label{PDE}
R_{g_j} &=-(n-1)e^{-2f_j} \lp 2\Delta^{g_{0,j}} f_j + (n-2)|\nabla^{g_{0,j}} f_j|^2\rp \ge -\frac{1}{j}  \hspace{0.5cm}
\\&\qquad \Rightarrow \hspace{0.5cm}2\Delta^{g_{0,j}} f_j + (n-2)|\nabla^{g_{0,j}} f_j|^2 \le \frac{e^{2f_j}}{(n-1)j}
\end{align}
and now by integrating this PDE we find that
\begin{align}
\int_{T^n}|\nabla^{g_{0,j}} f_j|^2 dV_{g_{0,j}} \le \frac{1}{(n-1)j} \int_{T^n} e^{2f_j} d V_{g_{0,j}}.
\end{align}

By H\"{o}lder's inequality we find for $n \ge 2$
\begin{align}
\int_{\Tor^n} e^{2f_j} d V_{g_{0,j}} &\le \left(\int_{\Tor^n} e^{nf_j} d V_{g_{0,j}}\right)^{\frac{2}{n}} \Vol(M_{0,j})^{\frac{n-2}{n}}
\\&\le \Vol(M_j)^{\frac{2}{n}}\Vol(M_{0,j})^{\frac{n-2}{n}} \le V_0^{\frac{2}{n}}\Vol(M_{0,j})^{\frac{n-2}{n}},
\end{align}
which implies \eqref{FirstSobolevEstimate}.

Now we also notice
\begin{align}\label{ProductRuleExponential}
-(n-1)\Delta^{g_{0,j}} e^{-2f_j} = 2(n-1)\Delta^{g_{0,j}} f_j e^{-2f_j} -4(n-1) |\nabla^{g_{0,j}} f_j|^2 e^{-2f_j}
\end{align}
which implies
\begin{align}
-(n-1)\Delta^{g_{0,j}} e^{-2f_j} + (n-1)(n+2) |\nabla^{g_{0,j}} e^{-f_j}|^2  \le \frac{1}{j}.
\end{align}\label{GoToZeroExponential}
By integrating we immediately find that 
\begin{align}
\int_{\mathbb{T}^n} |\nabla^{g_{0,j}} e^{-f_j}|^2  dV_{g_{0,j}} \le \frac{\Vol(M_{0,j})}{j(n-1)(n+2)}.
\end{align}
Similarly we can find
\begin{align}
\frac{2}{\alpha} \Delta^{g_{0,j}} e^{\alpha f_j} -2 \alpha |\nabla^{g_{0,j}} f_j|^2 e^{\alpha f_j} = 2 \Delta^{g_{0,j}} f_j e^{\alpha f_j}\label{AlphaLaplacianIdentity}
\end{align}
and so by combining with \eqref{PDE} we find
\begin{align}
\frac{2}{\alpha} \Delta^{g_{0,j}} e^{\alpha f_j} +(n - 2-2 \alpha) |\nabla^{g_{0,j}} f_j|^2 e^{\alpha f_j} \le \frac{e^{(2+\alpha) f_j}}{(n-1)j}.\label{GenPDEInequality}
\end{align}
We can rewrite \eqref{GenPDEInequality} as
\begin{align}
\frac{2}{\alpha} \Delta^{g_{0,j}} e^{\alpha f_j} +\frac{4(n - 2-2 \alpha)}{\alpha^2} |\nabla^{g_{0,j}} e^{\frac{\alpha f_j}{2}}|^2  \le \frac{e^{(2+\alpha) f_j}}{(n-1)j},
\end{align}
and hence by integrating we find
\begin{align}
\int_{\mathbb{T}^n}|\nabla^{g_{0,j}} e^{\frac{\alpha}{2}f_j}|^2dV_{g_{0,j}} &\le \frac{\alpha^2}{4(n-2-2\alpha)(n-1)j}\int_{\Tor^n}e^{(2+\alpha)f_j}dV_{g_{0,j}}
\end{align}
 if $\alpha \in \lp 0,\frac{n-2}{2}\rp$. By this choice of $\alpha$ we find for $n \ge 2$
 \begin{align}
 2+\alpha < \frac{n+2}{2} \le n,
 \end{align}
 and hence by H\"{o}lder's inequality 
 \begin{align}
 \int_{\Tor^n}e^{(2+\alpha)f_j}dV_{g_{0,j}} &\le\left( \int_{\Tor^n}e^{nf_j}dV_{g_{0,j}}\right)^{\frac{2+\alpha}{n}} \Vol(M_{0,j})^{\frac{n-2-\alpha}{n}}
 \\ &\le \Vol(M_j)^{\frac{2+\alpha}{n}} \Vol(M_{0,j})^{\frac{n-2-\alpha}{n}} \le V_0^{\frac{2+\alpha}{n}} \Vol(M_{0,j})^{\frac{n-2-\alpha}{n}},
 \end{align}
 which implies \eqref{ThirdSobolevEstimate}.
\end{proof}
\section{Reduction To The Conformal Case}\label{sect:Reduction}

In this section we consider $M_j^n=(\Tor^n,g_j)$ which is diffeomorphic to a torus and satisfies the hypotheses of Conjecture \ref{MainConjecture}. Our goal is to relate this conjecture to a sequence of metrics conformal to background metrics with constant scalar curvature. To this end we remember the statement of the Yamabe problem (See the paper by Lee and Parker for an extensive discussion of the Yamabe problem \cite{LeePar87}).

\begin{thm}[Yamabe \cite{Yamabe60}, Trudinger \cite{Trudinger68}, Aubin \cite{Aubin76}, Schoen \cite{Schoen84}]\textbf{The Yamabe Problem Theorem:} \label{Yamabe Problem}
If $M^n$ is a compact manifold of dimension $n \ge 3$ with Riemannian metric $g$ then there exists a constant scalar curvature metric $g_0$ and a function $f$ such that $g = e^{2f} g_0$. In other words, $g$ is conformal to a Riemannian manifold with constant scalar curvature.
\end{thm}

This implies that $g_j$ is conformal to a metric with constant scalar curvature on the torus, often called Yamabe metrics. Note that if $g_j = e^{2f_j}g_0$ and $R_{g_0} = C\not = 0$ then one can define $\tilde{g}_0 = |C|g_0$ so that $R_{\tilde{g}_0} = \frac{R_{g_0}}{|C|}=\frac{C}{|C|}$. Hence we can assume without loss of generality that $R_{g_0} = -1,0,1$. 

Now we remember the scalar torus rigidity theorem.
\begin{thm}[Schoen and Yau \cite{SY79}, Gromov and Lawson \cite{GL80}]\label{ScalarTorusRigidity}
If $M^n$ is a Riemannian $n-$manifold, $n\ge 2$ which is diffeomorphic to the torus with non-negative scalar curvature then $M$ is isometric to the flat torus.
\end{thm}
This shows that in our case we can eliminate the possibility of $R_{g_0}$ being positive and hence we can assume that $R_{g_0} = -1,0$. 

Now we will eliminate the possibility that the sequence $g_j$ in Conjecture \ref{MainConjecture} can by conformal to a sequence of metrics with constant scalar curvature $-1$ whose volume does not go to zero fast enough.

\begin{prop}\label{UniformFlatTori}
If $M_j^n=(M^n,g_j)$ is a sequence of Riemannian $n-$manifolds, $n \ge 3$ with scalar curvature 
\begin{align}
R_{g_j} \ge -\frac{1}{j}
\end{align}
 and volume 
 \begin{align}
 \Vol(M_j) \le V_0
 \end{align}
  then if $M_j$ is conformal to a sequence $M_{0,j}=(M,g_{0,j})$, i.e. $g_j = e^{2f_j} g_{0,j}$, with scalar curvature 
  \begin{align}
  R_{g_{0,j}} = -1
  \end{align}
  then
  \begin{align}
   \Vol(M_{0,j}) \le \frac{V_0}{j^{\frac{n}{2}}}.\label{ContradictionVolumeAssumption}
   \end{align}
\end{prop}
\begin{proof}
Assume on the contrary that $M_j$ is a Riemannian $n-$manifold, $n \ge 3$ with scalar curvature $R_{g_j} \ge -\frac{1}{j}$ that is conformal to the sequence $g_{0,j}$ with scalar curvature $R_{g_{0,j}} \equiv -1$ and 
\begin{align}
 \Vol(M_{0,j}) > \frac{V_0}{j^{\frac{n}{2}}}.
 \end{align}
  So there exists a sequence $f_j \in C^2(M)$ such that $g_j = e^{2f_j} g_{0,j}$ and by \eqref{ScalarFormula} we find
\begin{align}
R_{g_j} &= e^{-2f_j} \left(R_{g_{0,j}} - 2(n-1)\Delta^{g_{0,j}} f_j -(n-2)(n-1)|\nabla^{g_{0,j}} f_j|^2 \right) \ge -\frac{1}{j}.
\end{align}

By applying our assumptions we find
\begin{align}
e^{-2f_j} \left(-1 - 2(n-1)\Delta^{g_{0,j}} f_j -(n-2)(n-1)|\nabla ^{g_{0,j}}f_j|^2 \right) \ge -\frac{1}{j}
\end{align}
and be rearranging we find
\begin{align}
 2(n-1)\Delta^{g_{0,j}} f_j +(n-2)(n-1)|\nabla^{g_{0,j}} f_j|^2  \le \frac{e^{2f_j}}{j}-1.
\end{align}

Now by integrating this inequality over $M_{0,j}$ we find
\begin{align}\label{CrucialReductionFormula}
(n-2)(n-1) \int_{M} |\nabla^{g_{0,j}} f_j|^2 dV_{g_{0,j}} \le \frac{1}{j} \int_{M} e^{2f_j}dV_{g_{0,j}} - Vol(M_{0,j}).
\end{align}
Recall that $Vol(M_j) = \int_{M_{0,j}} e^{nf_j}dV$ and by H\"{o}lder's inequality with $p = n/2$ and $q = \frac{n}{n-2}$
\begin{align}\label{HolderConsequenceReduction}
\int_{M} e^{2f_j}dV_{g_{0,j}} &\le Vol(M_{0,j})^{\frac{n-2}{n}}\left(\int_{M}e^{nf_j}dV_{g_{0,j}} \right)^{\frac{2}{n}} 
\\&= Vol(M_{0,j})^{\frac{n-2}{n}}Vol(M_j)^{\frac{2}{n}}  \le Vol(M_{0,j})^{\frac{n-2}{n}}V_0^{\frac{2}{n}}.
\end{align}
Now by combining this inequality with \eqref{CrucialReductionFormula} we find
\begin{align}
(n-2)(n-1) \int_{M_{0,j}} |\nabla^{g_{0,j}} f_j|^2 dV_{g_{0,j}} &\le \frac{1}{j}Vol(M_{0,j})^{\frac{n-2}{n}}V_0^{\frac{2}{n}} - Vol(M_{0,j}) 
\\&= Vol(M_{0,j})^{\frac{n-2}{n}} \left( \frac{V_0^{\frac{2}{n}}}{j} - Vol(M_{0,j})^{2/n}\right).\label{SignConsequenceInequality}
\end{align}

Then by \eqref{SignConsequenceInequality} and \eqref{ContradictionVolumeAssumption} we find
\begin{align}
(n-2)(n-1) \int_{M} |\nabla^{g_{0,j}} f_j|^2 dV_{0,j} &\le Vol(M_{0,j})^{\frac{n-2}{n}} \left( \frac{V_0^{\frac{2}{n}}}{j} - \left(\frac{V_0}{j^{\frac{n}{2}}}\right)^{\frac{2}{n}}\right) < 0\label{ContradictionInequality1}
\end{align}
which leads to a contradiction since the left hand side of \eqref{ContradictionInequality1} is non-negative and the right hand side is negative for large $j$.

\end{proof}

The last theorem allowed us to rule out an important case but we still have the possibility that $M_{0,j}$ can have Vol$(M_{0,j}) \rightarrow 0$. Since there are metrics of constant negative scalar curvature whose volume goes to zero this case needs to be dealt with separately. We will not be able to handle this case entirely in this paper but we note that this provides one possibly approach to proving Conjecture \ref{MainConjecture} in general which was discussed in Remark \ref{remark:MainThm2Discussion}.

 We now define the conformal class of a Riemannian metric and the Yamabe constant for the conformal class of a Riemannian metric.

\begin{defn}
Let $g$ be a Riemannian metric on $M$ and define the conformal class
\begin{align}
[g] = \{g' \text{ Riemannian metric }: g' = u^{\frac{4}{n-2}} g, \qquad u \in C^{\infty}(M)\},
\end{align}
and the Yamabe constant
\begin{align}
\mu(M,[g]) = \inf_{u \in W^{1,2}(M), u > 0} \frac{ \left(\int_M 4\frac{n-1}{n-2} |\nabla u|_g^2 + R_g u^2 dV\right)}{\left(\int_M u^{\frac{2n}{n-2}} dV \right)^{\frac{n-2}{n}}}.
\end{align}
\end{defn}

Note that by the Theorem \ref{Yamabe Problem} we know that $\mu(M,g)$ is realized by a constant scalar curvature equal to $\mu(M,g)$ metric for compact manifolds $M$. We now see that for sequences of Yamabe metrics that converge in $C^0$ and whose scalar curvatures converge must have the Yamabe constants of their conformal classes converge.

\begin{thm}[B. Bergery \cite{Ber81}]\label{Maximizers of Yamabe}
Let $M_i=(M,g_i)$ be a sequence of smooth metrics with scalar curvature $R_i$ and conformal class $[g_i]$. Assume that $g_i \rightarrow g_{\infty}$ in $C^0$ and $R_i \rightarrow R_{\infty}$ in $C^0$. Then, 
\begin{align}
\lim_{i \rightarrow \infty} \mu(M,[g_i]) = \mu(M,[g_{\infty}]).
\end{align} 

\end{thm}

We can now use this result to see that if we have a background sequence of Yamabe metrics that converge to a Riemannian metric whose scalar curvatures converge to zero then they must converge to a flat torus. This justifies the assumption in Theorem \ref{MainThm3}.

\begin{cor}\label{Implies Convergence to Flat Tori}
For a sequence of Riemannian manifolds $\tilde{M}_{0,j}=(M,\tilde{g}_{0,j})$ with constant negative scalar curvature $R_{g_{0,j}} \rightarrow 0$ and unit volume, so that
\begin{align}
\tilde{g}_{0,j} \rightarrow g_0 \text{ in } C^0,
\end{align}
 where $g_0$ is a Riemannian metric on $\mathbb{T}^n$ then $g_0$ is a flat torus where $\mathbb{T}^n_0=(\mathbb{T}^n,g_0)$, so that if $\theta_i \in [0,2\pi], 1 \le i \le n,$ are coordinates on $\Tor^n$ then
\begin{align}
0<c\le|(g_0)_{ik}|\le C<\infty, \quad (g_0)_{ik} \in \R,\quad 1 \le i,k\le n,
\end{align}

\end{cor}
\begin{proof}
By Theorem \ref{Maximizers of Yamabe} we know that 
\begin{align}
\lim_{i \rightarrow \infty} \mu(\mathbb{T}^n,[\tilde{g}_{0,j}]) = \mu(\mathbb{T}^n,[g_0]).
\end{align} 
Now since we know that $R_{g_{0,j}}$ is a constant which converges to $0$ we have that
\begin{align}\label{Maximizes Yamabe Constant}
\lim_{i \rightarrow \infty} \mu(\mathbb{T}^n,[\tilde{g}_{0,j}]) = \mu(\mathbb{T}^n,[g_0]) = 0,
\end{align} 
which implies that $g_0$ must be an Einstein metric on $\mathbb{T}^n$. Hence $g_0$ must have constant zero sectional curvature which implies that $g_0$ is the metric of a flat torus on $\mathbb{T}^n$ (See Proposition 1.3 of Schoen \cite{Schoen87}).
\end{proof}

See chapter 4 and 12 of A. Besse's book \cite{Bess87} for a further discussion of topics related to Theorem \ref{Maximizers of Yamabe} and Corollary \ref{Implies Convergence to Flat Tori}. A comparison argument by Aubin \cite{Aubin76} implies that for any conformal class $[g]$ on $M$ the Yamabe constant can be uniformly bounded above by the Yamabe constant of the round sphere
\begin{align}
\mu(M,[g]) \le \mu(\mathbb{S}^n, [g_0]),
\end{align}
where $g_0$ is the round metric of unit volume on $\mathbb{S}^n$. This means that the suprememum over conformal classes is well defined which motivates the following definition.

\begin{defn}
The Sigma constant of a manifold $M$ is
\begin{align}
\sigma(M)=\sup_{[g]\in \mathcal{C}} \mu(M,[g]),
\end{align}
where $\mathcal{C}$ is the set of all conformal structures on $M$.
\end{defn}

The Sigma constant of a torus is zero, $\sigma(\mathbb{T}^n)=0$, which is realized by flat tori which are clearly not unique. In section 4 of Anderson \cite{Anderson97} the author discusses maximizing sequences of Yamabe metrics which realize the sigma constant of a manifold. Three important conjectures are made about the convergence of such sequences in the case where $\sigma(M) < 0$, $\sigma(M) = 0$, or $\sigma(M) > 0$. In the case where $M=\mathbb{T}^n$ it is conjectured that if the sequence of Yamabe metrics realizing $\sigma(\mathbb{T}^n)=0$ does not degenerate then it will converge to a flat torus. What convergence is expected depends on extra conditions one imposes on the sequence and the topology one considers on the space of Yamabe metrics. In Theorem \ref{MainThm3} we assume $C^1$ convergence of the maximizing sequence of Yamabe metrics which is used in the proof.  It is then further noted that a sequence of Yamabe metrics realizing $\sigma(\mathbb{T}^n)=0$ could degenerate to a lower dimensional flat torus $\mathbb{T}^k$, $0 < k < n$ or could degenerate to a point. In the current paper we do not deal with these cases of degeneration but it is the subject of further study. These cases make the analysis quite delicate and are expected to be rather difficult to deal with. In general, further study of maximizing sequences of Yamabe metrics and their convergence is interesting in its own right and further motivated by the results of this paper.

\section{Sequences Conformal to Flat Tori}\label{sect:ConfFlatCase}

In this section we study sequences which satisfy Conjecture \ref{MainConjecture} and which are conformal to a flat torus.  Throughout this section we will use the important results of  Lemma \ref{SobolevConvergence1} derived in section \ref{sec: background}.

\subsection{$C^0$ Convergence from Below}\label{subsec:C^0 Convergence}

Now we would like to study the equations of Lemma \ref{SobolevConvergence1}. By using an auxiliary function we will be able to apply the mean value property for subharmonic functions to obtain $C^0$ control from below on the conformal factors satisfying the hypotheses of Theorem \ref{MainThm}.

\begin{lem}\label{C0Control1}For $e^{-2f_j}$ a solution to \begin{align}\label{PDEInequality}
-(n-1)\Delta^{g_{0,j}} e^{-2f_j}  \le  -(n-1)\Delta^{g_{0,j}} e^{-2f_j} + (n-1)(n+2) |\nabla^{g_{0,j}} e^{-f_j}|^2  \le \frac{1}{j},
\end{align} 
where the equation is defined with respect to $g_{0,j}$ so that if $\theta_i \in [0,2\pi], 1 \le i \le n,$ are coordinates on $\Tor^n$ we find
\begin{align}
0<c\le|(g_{0,j})_{ik}|\le C<\infty, \quad (g_{0,j})_{ik} \in \R,\quad 1 \le i,k\le n,
\end{align}
then there exists an $R>0$ such that
\begin{align}
\lp\limsup_{j \rightarrow \infty}\max_{x \in \mathbb{T}^n}\dashint_{B_{M_{0,j}}(x,r)} e^{-2f_j} dV_{g_{0,j}}\rp^{-1} \le \liminf_{j \rightarrow \infty} \min_{\mathbb{T}^n} e^{2f_j}, \qquad r \le R.
\end{align}
\end{lem}
\begin{proof}
Let $(\theta_1,\theta_2,...,\theta_n)$, $ \theta_i \in [0,2\pi]$, $1 \le i \le n$ be coordinates on $\mathbb{T}^n$. Then consider $ e^{\alpha \theta_1}$, $\alpha > 0$ and compute
\begin{align}
\Delta^{g_{0,j}}(e^{\alpha \theta_1}) = \frac{\alpha^2}{(g_{0,j})_{11}^2}e^{\alpha \theta_1}  \ge  \frac{\alpha^2}{(g_{0,j})_{11}^2}
\end{align}
Then if we consider $e^{-2f_j} +e^{\alpha \theta_1}$ then we find
\begin{align}
-\Delta^{g_{0,j}}(e^{-2f_j} +e^{\alpha \theta_1}) \le \frac{1}{(n-1)j} -\frac{\alpha^2}{(g_{0,j})_{11}^2} =0\label{subharmonicEq}
\end{align}
where we choose $\alpha = \frac{(g_{0,j})_{11}}{\sqrt{(n-1)j}}$.
So  $e^{-2f_j}+e^{\alpha \theta_1}$ is a solution to \eqref{subharmonicEq} it follows from the mean value property for subharmonic functions that if one chooses an $R>0$ so that $B_{M_{0,j}}((\pi,\pi,...,\pi),r) \subset [0,2\pi]^n$ for all $j\in \N$ and $0<r<R$ then
\begin{align}
e^{-2f_j(\pi,\pi,...,\pi)} \le \dashint_{B_{M_{0,j}}((\pi,\pi,...,\pi),r)} e^{-2f_j} dV_{g_{0,j}} + e^{\frac{2\pi (g_{0,j})_{11}}{\sqrt{(n-1)j}}} , \qquad r \le R.
\end{align}
Note that by recentering the coordinate system on $\mathbb{T}^n$ we can repeat this argument for any $x \in \mathbb{T}^n$.
Hence this implies
\begin{align}
\limsup_{j \rightarrow \infty}\max_{\Tor^m} e^{-2f_j} \le \limsup_{j \rightarrow \infty}\max_{x \in \mathbb{T}^n}\dashint_{B_{M_{0,j}}(x,r)} e^{-2f_j} dV_{g_{0,j}}, \qquad r \le R.
\end{align}
The statement in the theorem follows by rearranging terms and using properties of reciprocals.
\end{proof}

In order for Lemma \ref{C0Control1} to be useful we will need to control $\dashint_{B_{M_{0,j}}(x,r)} e^{-2f_j} dV_{g_{0,j}}$ and we begin by gaining control of $\dashint_{\Tor^m} e^{-2f_j} dV_{g_{0,j}}$ in Lemma \ref{e^(-2f_j) Control Summary}.

\begin{lem}\label{e^(-2f_j) Control Summary}
Let $M_j=(\mathbb{T}^n,g_j)$ be a sequence of Riemannian manifolds such that $g_j = e^{2f_j} g_{0,j}$ where $M_{0,j}=(\mathbb{T}^n,g_{0,j})$ is a flat torus and
\begin{align}
\int_{\mathbb{T}^n} e^{-2f_j} d V_{g_{0,j}} \le C.
\end{align} 
Then we find that
\begin{align}
\frac{\Vol(M_{0,j})^{\frac{n+2}{n}}}{\Vol(M_j)^{\frac{2}{n}}}\le \int_{\mathbb{T}^n} e^{-2f_j} d V_{g_{0,j}} &\le C,
\\\frac{\Vol(M_{0,j})^{\frac{n+1}{n}}}{\Vol(M_j)^{\frac{1}{n}}}\le \int_{\mathbb{T}^n} e^{-f_j} d V_{g_{0,j}} &\le \sqrt{C}\Vol(M_{0,j})^{\frac{1}{2}} .
\end{align}
\end{lem}
\begin{proof}
The upper bound on $\int_{\mathbb{T}^n} e^{-2f_j} d V_{g_{0,j}}$ follows by assumption. The lower bound follows from Jensen's Inequality since $\varphi(x) = \frac{1}{x}$ is a convex functions and hence
\begin{align}
\dashint_{\mathbb{T}^n} e^{-2f_j} d V_{g_{0,j}} &\ge \frac{1}{\dashint_{\mathbb{T}^n} e^{2f_j} d V_{g_{0,j}}},
\\\dashint_{\mathbb{T}^n} e^{-f_j} d V_{g_{0,j}} &\ge \frac{1}{\dashint_{\mathbb{T}^n} e^{f_j} d V_{g_{0,j}}},
\end{align}
which implies
\begin{align}
\int_{\mathbb{T}^n} e^{-2f_j} d V_{0,j} &\ge \frac{\Vol(M_{0,j})^2}{\int_{\mathbb{T}^n} e^{2f_j} d V_{g_{0,j}}},\label{JensenConsequence}
\\\int_{\mathbb{T}^n} e^{-f_j} d V_{0,j} &\ge \frac{\Vol(M_{0,j})^2}{\int_{\mathbb{T}^n} e^{f_j} d V_{g_{0,j}}}.\label{JensenConsequence2}
\end{align}
By H\"{o}lder's inequality
\begin{align}
\int_{T^n} e^{2f_j} d V_{g_{0,j}} &\le \left(\int_{T^n} e^{nf_j} d V_{g_{0,j}}\right)^{\frac{2}{n}} \Vol(M_{0,j})^{\frac{n-2}{n}}
\le \Vol(M_j)^{\frac{2}{n}}\Vol(M_{0,j})^{\frac{n-2}{n}} ,
\\\int_{T^n} e^{f_j} d V_{g_{0,j}} &\le \left(\int_{T^n} e^{nf_j} d V_{g_{0,j}}\right)^{\frac{1}{n}} \Vol(M_{0,j})^{\frac{n-1}{n}}
\le \Vol(M_j)^{\frac{1}{n}}\Vol(M_{0,j})^{\frac{n-1}{n}} ,
\end{align}
which when combined with \eqref{JensenConsequence} and \eqref{JensenConsequence2} we find
\begin{align}
\int_{\mathbb{T}^n} e^{-2f_j} d V_{g_{0,j}} &\ge \frac{\Vol(M_{0,j})^{\frac{n+2}{n}}}{\Vol(M_j)^{\frac{2}{n}}},
\\\int_{\mathbb{T}^n} e^{-f_j} d V_{g_{0,j}} &\ge \frac{\Vol(M_{0,j})^{\frac{n+1}{n}}}{\Vol(M_j)^{\frac{1}{n}}}.
\end{align}
The upper bound on $\int_{\mathbb{T}^n} e^{-f_j} d V_{g_{0,j}}$ again follows from H\"{o}lder's inequality.
\end{proof}

\subsection{Volume Convergence and $C^0$ Convergence from Below}\label{subsec:L^3 Conv}

We note that by the upper bound on volume hypothesis there must be a subsequence whose volume converges to a finite number. Now we need to ensure that the volume convergence agrees with the $C^0$ convergence from below of subsection \ref{subsec:C^0 Convergence}. This is done by combining  Corollary \ref{e^(-2f_j) Control Summary} with Lemma \ref{SobolevConvergence1} to give a relationship between volume convergence and the $C^0$ convergence from below.

\begin{prop}\label{L^3 Agrees with C^0 2}
Let $M_j$ be a sequence of Riemannian manifolds such that $g_j = e^{2f_j} g_{0,j}$ where $M_{0,j}=(\mathbb{T}^n,g_{0,j})$ is a flat torus 
so that if $\theta_i \in [0,2\pi], 1 \le i \le n,$ are coordinates on $\Tor^n$ we find
\begin{align}
0<c\le|(g_{0,j})_{ik}|\le C<\infty, \quad (g_{0,j})_{ik} \in \R,\quad 1 \le i,k\le n,
\end{align}
\begin{align}
\Vol(M_j)\le V_0,
\end{align}
\begin{align}
\int_{\mathbb{T}^n} e^{-2f_j} d V_{g_{0,j}} \le C,
\end{align}
and
 \begin{align}
\exists q > 0 \text{ so that } \forall E \subset \mathbb{T}^n \text{ measurable } \Vol_{g_j}(E) \le C \Vol_{g_{0,j}}(E)^q,
 \end{align}
 Then  there exists a subsequence so that
\begin{align}
\lim_{k \rightarrow \infty}\dashint_{\Tor^n} e^{nf_k} dV_{g_{0,k}} = \lim_{k \rightarrow \infty} \left (\dashint_{\Tor^n} e^{-2f_k} dV_{g_{0,k}} \right)^{-n/2} = \lim_{k \rightarrow \infty}(\overline{e^{-f_k}})^{-n},
\end{align}
where 
\begin{align}
\overline{e^{-f_k}} = \dashint_{\Tor^n}e^{-f_k}dV_{g_{0,k}},
\end{align}
and there exists an $R>0$ such that
\begin{align}
 \forall r \le R \quad \lim_{k \rightarrow \infty}\max_{x \in \mathbb{T}^n} \left (\dashint_{B_{M_{0,k}}(x,r)} e^{-2f_k} dV_{g_{0,k}} \right) = \lim_{k \rightarrow \infty}(\overline{e^{-f_k}})^{2}.
\end{align}
\end{prop}
\begin{proof}
We first note that 
 \begin{align}
 \Vol_{g_j}(E) \le C \Vol_{g_{0,j}}(E)^q,
 \end{align}
implies
 \begin{align}
 \int_{E}e^{n f_j} dV_{g_{0,j}} \le C \Vol_{g_{0,j}}(E)^{q},
 \end{align}
 which is uniform integrability for the function $e^{nf_j}$ and hence when combined with the upper bound on volume implies a subsequence converges weakly in $L^1$.
 
Lemma \ref{SobolevConvergence1} and the Poincare inequality imply
\begin{align}\label{Poincare Consequence}
C_P \int_{\Tor^n}|e^{-f_j} - \overline{e^{-f_j}}|^2dV_{g_{0,j}} \le \int_{\Tor^n} |\nabla^{g_{0,j}} e^{-f_j}|^2  dV_{g_{0,j}} &\le \frac{Vol(M_{0,j})}{j(n-1)(n+2)},
\end{align}
where the constant in the Poincare inequality is uniform by the uniform bounds on the background sequence of flat tori $M_{0,j}$.
\eqref{Poincare Consequence} implies $L^2$ convergence of $e^{-f_j}$ to its average as well as pointwise a.e. convergence on a subsequence of $e^{-f_k}$ to its average,
\begin{align}
\overline{e^{-f_k}} = \dashint_{\Tor^n}e^{-f_j}dV_{g_{0,k}},
\end{align}
which is non-zero and well defined by Corollary \ref{e^(-2f_j) Control Summary} and where we have taken a further subsequence if necessary so that $\displaystyle \lim_{k \rightarrow \infty} \overline{e^{-f_k}}$ is well defined. This also implies convergence in $L^2$ norm for any measurable set $U \subset \mathbb{T}^n$
\begin{align}\label{ImportantConvergenceEq}
\lim_{k\rightarrow \infty}\int_{U} e^{-2f_k} dV_{g_{0,k}}=\lim_{k\rightarrow \infty} \left(\overline{e^{-f_k}}\right)^2\Vol_{g_{0,k}}(U).
\end{align} 
In fact, by the reverse triangle inequality for norms, if we have two functions $h,k:\mathbb{T}^n\rightarrow \R$ we find
\begin{align}
\left|\|h\|_{L^2_g(U)}-\|k\|_{L^2_g(U)}\right| \le \|h-k\|_{L^2_g(U)} \le \|h-k\|_{L^2_g(\mathbb{T}^n)},
\end{align}
for any Riemannian metric $g$ and hence \eqref{ImportantConvergenceEq} is uniform in $U$ and in particular uniform in $B_{M_{0,j}}(x,r)$ for all $x \in \mathbb{T}^n$, $0 < r \le R$.

We also note that we have $e^{f_k} \rightarrow\displaystyle\lim_{k\rightarrow \infty}(\overline{e^{-f_k}})^{-1}$ a.e. and hence by Egeroff's theorem one finds that for any $\varepsilon > 0$ there exists a measurable set $E \subset \mathbb{T}^n$ so that $e^{f_k}$ converges uniformly on $E$ and $dV_{g_{0,\infty}}(\mathbb{T}^n\setminus E) < \epsilon$. Hence we find
\begin{align}
\left|\int_{\mathbb{T}^n} (e^{nf_k}-(\overline{e^{-f_k}})^{-n})dV_{g_{0,k}} \right|&\le \left|\int_{E} (e^{nf_k}-(\overline{e^{-f_k}})^{-n})dV_{g_{0,k}}\right|
\\&\quad +\left|\int_{\mathbb{T}^n\setminus E} (e^{nf_k}-(\overline{e^{-f_k}})^{-n})dV_{g_{0,k}}\right|
\\&\le \Vol_{g_{0,k}}(E) \|e^{nf_k}-(\overline{e^{-f_k}})^{-n})\|_{C^0(E)}
\\&\quad + C\Vol_{g_{0,k}}(\mathbb{T}^n\setminus E)^q+(\overline{e^{-f_k}})^{-n} \Vol_{g_{0,k}}(\mathbb{T}^n\setminus E) 
\\&\le C' \left( \frac{1}{k}+\varepsilon^q + \varepsilon\right).
\end{align}
This implies volume convergence to the volume of the limit of the sequence $(\overline{e^{-f_k}})^{-2}g_{0,k}$ where again a subsequence can be taken to make this limit well defined.

To summarize we have found
\begin{align}
\lim_{k \rightarrow \infty} \max_{\mathbb{T}^n}\dashint_{B_{M_{0,k}}(x,r)} e^{-2f_k} dV_{g_{0,k}}= \lim_{k \rightarrow \infty}(\overline{e^{-f_k}})^2,\quad r \le R
\end{align}
and 
\begin{align}
\lim_{k \rightarrow \infty}\int_{\Tor^n} e^{nf_k} dV_{g_{0,k}} = \lim_{k \rightarrow \infty} \int_{\Tor^n} (\overline{e^{-f_k}})^{-n} dV_{g_{0,k}},
\end{align}
which implies the desired result.
\end{proof}

We now prove a standard lemma which shows that the uniform integrability condition is implied by $L^p$ bounds for $p > n$.

\begin{lem}\label{L^3 Agrees with C^0 2}
Let $M_j$ be a sequence of Riemannian manifolds such that $g_j = e^{2f_j} g_{0,j}$ where $M_{0,j}=(\mathbb{T}^n,g_{0,j})$ is a flat torus. If 
\begin{align}
\exists \, p > n \quad \text{ so that }\quad \int_{\Tor^n} e^{pf_j} d V_{g_{0,j}} \le C_p,
\end{align}
then $\forall E \subset \mathbb{T}^n$ measurable
 \begin{align}
 \Vol_{g_j}(E) \le C_p \Vol_{g_{0,j}}(E)^{\frac{p-n}{p}}.
 \end{align}
\end{lem}
\begin{proof}
Apply H\"{o}lder's inequality with $l=\frac{p}{n}$ and $q=\frac{p}{p-n}$ to find
\begin{align}
 \Vol_j(E)&= \int_{E} e^{nf_j} dV_{g_{0,j}}
 \\&\le \Vol_{g_{0,j}}(E)^{\frac{p-n}{p}} \int_{E} e^{pf_j} dV_{g_{0,j}}
  \\&\le \Vol_{g_{0,j}}(E)^{\frac{p-n}{p}} \int_{\Tor^n} e^{pf_j} dV_{g_{0,j}} \le  C_p\Vol_{g_{0,j}}(E)^{\frac{p-n}{p}}.
\end{align}
\end{proof}

\subsection{SWIF Convergence to a Flat Tori}\label{subsec:SWIF Conv }

In this subsection we finish the proof of Theorem \ref{MainThm} and Theorem \ref{MainThm3} by noticing that we have the necessary hypotheses to apply Theorem \ref{VolSWIFConv} of the author, Perales, and Sormani.

\begin{proof}[Proof of Theorem \ref{MainThm}]
First we notice that if $\theta_i \in [0,2\pi], 1 \le i \le n,$ are coordinates on $\Tor^n$ so that
\begin{align}
0<c\le|(g_{0,j})_{rs}|\le C<\infty, \quad (g_{0,j})_{rs} \in \R,\quad 1 \le r,s\le n,
\end{align} 
then there exists a subsequence so that
\begin{align}
\lim_{k \rightarrow \infty} (g_{0,k})_{rs} = (g_{\infty})_{rs},\quad 1 \le r,s\le n,
\end{align}
where
\begin{align}
0<C_1 \le |(g_{\infty})_{rs}| \le C_2<\infty,,\quad 1 \le r,s\le n,
\end{align}
and hence $g_{\infty}$ is a well defined flat torus on $\Tor^n$ so that
\begin{align}\label{FirstC0BelowEstimate}
\left(1 - \frac{C}{j}\right) g_{\infty}(v,v) \le g_{0,j}(v,v), \quad \forall v \in T_pM, p \in M.
\end{align}
We also note that
\begin{align}\label{SecondC0BelowEstimate}
\left(\liminf_{j \rightarrow \infty} \min_{\mathbb{T}^n}e^{2f_j}\right) g_{0,j}(v,v) \le e^{2f_j}g_{0,j}(v,v)=g_j(v,v),\quad \forall v \in T_pM, p \in M.
\end{align}
When we combine \eqref{FirstC0BelowEstimate} and\eqref{SecondC0BelowEstimate} with Lemma \ref{C0Control1} we find
\begin{align}\label{ThirdC0BelowEstimate}
\lp\limsup_{j \rightarrow \infty}\max_{x \in \mathbb{T}^n}\dashint_{B_{M_{0,j}}(x,r)} e^{-2f_j} dV_{g_{0,j}}\rp^{-1}\left(1 - \frac{C}{j}\right) g_{\infty}(v,v) \le g_j(v,v), \forall v \in T_pM, p \in M.
\end{align}
Lemma \ref{L^3 Agrees with C^0 2} gives the required volume convergence and Lemma \ref{e^(-2f_j) Control Summary} combined with \eqref{ThirdC0BelowEstimate}  implies that if $c_{\infty}=\displaystyle\lim_{k \rightarrow \infty}(\overline{e^{-f_k}})^{-2}$ then on a subsequence
\begin{align}\label{FourthC0BelowEstimate}
\left(1 - \frac{\bar{C}}{k}\right)\bar{g}_{\infty}=\left(1 - \frac{\bar{C}}{k}\right)c_{\infty} g_{\infty}(v,v) \le g_k(v,v), \quad \forall v \in T_pM, p \in M.
\end{align}
Again by Lemma \ref{L^3 Agrees with C^0 2} we know that on a subsequence \eqref{FourthC0BelowEstimate} agrees with the volume convergence and hence we have the hypotheses necessary in order to apply Theorem \ref{VolSWIFConv} to conclude that a subsequence of $M_j$ must converge to the flat torus $\bar{\Tor}^n_{\infty}$ in the volume preserving intrinsic flat sense.
\end{proof}

\begin{proof}[Proof of Theorem \ref{MainThm3}]
Since we have assumed that $\tilde{g}_{0,j} \rightarrow g_0$ in $C^1$ we note that 
\begin{align}\left(1 - \frac{C}{j}\right) g_0(v,v) \le \tilde{g}_{0,j}(v,v) &\le \left(1 + \frac{C}{j}\right) g_0(v,v), \quad \forall v \in T_pM, p \in M, \label{MetricBound}
\\ \int_{\mathbb{T}^n} e^{-2f_j} dV_{g_0} &\le C'.,\label{IntegralBound}
\end{align}
In addition, by \eqref{ScalarFormula} we find
\begin{align}
 2(n-1)\Delta^{\tilde{g}_{0,j}} f_j +(n-2)(n-1)|\nabla^{\tilde{g}_{0,j}} f_j|_{\tilde{g}_{0,j}}^2  \le \frac{e^{2f_j}}{j}+\tilde{R}_{0,j} \le \frac{e^{2f_j}}{j},
\end{align}
as well as,
\begin{align}\label{Essential PDE}
-(n-1)\Delta^{\tilde{g}_{0,j}} e^{-2f_j} + (n-1)(n+2) |\nabla^{\tilde{g}_{0,j}} e^{-f_j}|_{\tilde{g}_{0,j}}^2  \le \frac{1}{j} + \tilde{R}_{0,j} e^{-2f_j} \le \frac{1}{j}.
\end{align}
Now since $\tilde{g}_{0,j} \rightarrow g_0$ in $C^1$ we find
\begin{align}
\Delta^{\tilde{g}_{0,j}} e^{-2f_j} &= \Delta^{g_0} e^{-2f_j} + (\tilde{g}_{0,j}^{lm}-g_0^{lm}) \partial_l \partial_m e^{-2f_j} 
\\&+ \left[ \frac{1}{\sqrt{det(\tilde{g}_{0,j})}} \partial_l \left(\sqrt{det(\tilde{g}_{0,j})} \tilde{g}_{0,j}^{lm}\right) - \frac{1}{\sqrt{det(g_0)}} \partial_l \left(\sqrt{det(g_0)} g_0^{lm} \right) \right]\partial_m e^{-2f_j}
\\&= \left( 1 \pm \frac{C}{j} \right ) \Delta^{g_0} e^{-2f_j} \pm \frac{C_1}{j} |\nabla^{g_0} e^{-f_j}|_{g_0},
\end{align}
\begin{align}
|\nabla^{\tilde{g}_{0,j}} e^{-f_j}|_{\tilde{g}_{0,j}}^2 &= |\nabla^{g_0} e^{-f_j}|_{g_0}^2 + \left(\tilde{g}_{0,j}^{lm} - g_0^{lm}\right) \partial_l e^{-f_j} \partial_m e^{-f_j}
= \left( 1 \pm \frac{C}{j} \right )|\nabla^{g_0} e^{-f_j}|_{g_0}^2.
\end{align}
So we can rewrite \eqref{Essential PDE} as
\begin{align}\label{Essential PDE Rewrite}
-(n-1)\left( 1 \pm \frac{C}{j} \right ) \Delta^{g_0} e^{-2f_j} \mp 2 \frac{C_1}{j} |\nabla^{g_0} e^{-f_j}|_{g_0} + (n-1)(n+2) \left( 1 \pm \frac{C}{j} \right )|\nabla^{g_0} e^{-f_j}|_{g_0}^2 \le \frac{1}{j},
\end{align}
and then by using $ab \le \frac{1}{2} \left( a^2+b^2\right)$ we find
\begin{align}\label{Essential PDE Rewrite 2}
-(n-1)\left( 1 \pm \frac{C}{j} \right ) \Delta^{g_0} e^{-2f_j} -  \frac{C_1}{j} \left( 1 +|\nabla^{g_0} e^{-f_j}|_{g_0}^2 \right) + (n-1)(n+2) \left( 1 \pm \frac{C}{j} \right )|\nabla^{g_0} e^{-f_j}|_{g_0}^2 \le \frac{1}{j},
\end{align}
and by rewriting again we obtain
\begin{align}\label{Essential PDE Rewrite 3}
-(n-1)\left( 1 \pm \frac{C}{j} \right ) \Delta^{g_0} e^{-2f_j} +\left( (n-1)(n+2) \left( 1 \pm \frac{C}{j} \right )-\frac{C_1}{j} \right)|\nabla^{g_0} e^{-f_j}|_{g_0}^2 \le \frac{1+C_1}{j},
\end{align}
which for $j$ chosen large enough leads to the useful equation,
\begin{align}\label{Essential PDE Rewrite 2}
- \Delta^{g_0} e^{-2f_j}  &\le \frac{1+C_1}{(n-1)j\left( 1 \pm \frac{C}{j} \right )} \le \frac{C_2}{j}.
\end{align}
Lastly by integrating \eqref{Essential PDE Rewrite 3} for large enough $j$  we find that 
\begin{align}\label{SobolevBound}
\int_{\mathbb{T}^n} |\nabla^{g_0} e^{-f_j}|^2  dV_{g_0} \le \frac{(1+C_1)\Vol(M_0)}{j\left((n-1)(n+2)\left( 1 \pm \frac{C}{j} \right )-\frac{C_1}{j}\right)} \le \frac{C_3}{j}.
\end{align}
Hence all of the estimates of \ref{subsec:C^0 Convergence} and \ref{subsec:L^3 Conv} apply to $e^{2f_j}$ and $g_0$ since \eqref{MetricBound}, \eqref{IntegralBound}, \eqref{Essential PDE Rewrite 2}, and \eqref{SobolevBound} were the main tools used in these subsections combined with the uniform integrability, volume bounds, and diameter bound assumptions of Theorem \ref{MainThm3}. Thus we have the hypotheses necessary in terms of $g_0$ in order to apply Theorem \ref{VolSWIFConv} to conclude that a subsequence of $M_j$ must converge to a flat torus.
\end{proof}

\bibliography{ConfToriBib}{}
\bibliographystyle{plain}

\end{document}